\documentclass{amsart}
\usepackage{amsthm}
\usepackage{amsmath}
\usepackage{amssymb}
\usepackage{pstricks}
\usepackage{pst-node}
\usepackage{multido}
\usepackage{mathrsfs}
\usepackage{enumerate}
\usepackage{verbatim}
\newcommand{\complex}{\mathbb{C}}

\newcommand{\naturalnumbers}{\mathbb{N}}
\newcommand{\integers}{\mathbb{Z}}

\newcommand{\aalg}{\mathcal{A}}
\newcommand{\rring}{\mathcal{R}}

\newcommand{\zeroa}{0_{\aalg}}
\newcommand{\unita}{1_{\aalg}}
\newcommand{\gcrossedring}{\aalg \rtimes_\alpha^{\sigma} G}
\DeclareMathOperator{\aut}{Aut}
\DeclareMathOperator{\identity}{id}

\DeclareMathOperator{\comm}{Comm}

\DeclareMathOperator{\ann}{Ann}

\DeclareMathOperator{\supp}{supp}
\DeclareMathOperator{\sep}{Sep}
\DeclareMathOperator{\inn}{inn}
\newcommand{\ida}{\identity_{\aalg}}
\usepackage[all,cmtip]{xy}

\theoremstyle{plain}
\newtheorem{thm}{Theorem}[section]

\newtheorem{prop}[thm]{Proposition}
\newtheorem{cor}[thm]{Corollary}
\theoremstyle{definition}
\newtheorem{defn}[thm]{Definition}

\newtheorem{exmp}[thm]{Example}

\newtheorem{rem}[thm]{Remark}

\newtheorem{case}{Case}

\begin{document}

\title{Commutativity and ideals in algebraic crossed products}
\date{January 31, 2007}

\author{Johan {\"O}inert}
\address{Department of Mathematics and Computer Science\\
University of Antwerp, Middelheimlaan 1\\
B-2020 Antwerp, Belgium {\rm and} Centre for Mathematical
Sciences, Lund University,  Box 118, SE-221 00 Lund, Sweden}
\thanks{This work was partially supported by the Crafoord Foundation,
The Royal Physiographic Society in Lund, The Swedish Royal Academy of Sciences,
The Swedish Foundation of International Cooperation in Research and Higher Education (STINT)
and "LieGrits", a Marie Curie Research Training Network funded by the European Community as project MRTN-CT 2003-505078}
\email{johan.oinert@ua.ac.be}

\author{Sergei D. Silvestrov}
\address{Centre for Mathematical
Sciences, Lund University,  Box 118, SE-221 00 Lund, Sweden}
\email{sergei.silvestrov@math.lth.se}


\begin{abstract}
We investigate properties of commutative subrings
and ideals in non-commutative algebraic crossed products for actions
by arbitrary groups. A description of the commutant of the base coefficient
subring in the crossed product ring is given.
Conditions for commutativity and maximal commutativity of the
commutant of the base subring are provided in terms of the action as
well as in terms of the intersection of ideals in the crossed product
ring with the base subring, specially taking into account both the
case of base rings without non-trivial zero-divisors and the case of base rings with
non-trivial zero-divisors.
\end{abstract}

\maketitle

\section{Introduction}
The description of commutative subrings and commutative subalgebras
and of the ideals in non-commutative rings and algebras are
important directions of investigation for any class of
non-commutative algebras or rings, because it allows one to relate
representation theory, non-commutative properties, graded
structures, ideals and subalgebras, homological and other properties
of non-commutative algebras to spectral theory, duality, algebraic
geometry and topology naturally associated with the commutative
subalgebras. In representation theory, for example, one of the keys
to the construction and classification of representations is the
method of induced representations. The underlying structures behind
this method are the semi-direct products or crossed products of
rings and algebras by various actions. When a non-commutative ring
or algebra is given, one looks for a subring or a subalgebra such
that its representations can be studied and classified more easily,
and such that the whole ring  or algebra can be decomposed as a
crossed product of this subring or subalgebra by a suitable action.
Then the representations for the subring or subalgebra are extended
to representations of the whole ring or algebra using the action and
its properties. A description of representations is most tractable
for commutative subrings or subalgebras as being, via the spectral
theory and duality, directly connected to algebraic geometry,
topology or measure theory.

If one has found a way to present a non-commutative ring or algebra
as a crossed product of a commutative subring or subalgebra by some
action on it of the elements from outside the subring or subalgebra,
then it is important to know whether this subring or subalgebra is
maximal abelian or, if not, to find a maximal abelian subring or
subalgebra containing the given subalgebra, since if the selected
subring or subalgebra is not maximal abelian, then the action will
not be entirely responsible for the non-commutative part as one
would hope, but will also have the commutative trivial part taking
care of the elements commuting with everything in the selected
commutative subring or subalgebra. This maximality of a commutative
subring or subalgebra and associated properties of the action are
intimately related to the description and classifications of
representations of the non-commutative ring or algebra.

Little is known in general about connections between properties of
the commutative subalgebras of crossed product rings and algebras
and properties of the action. A remarkable result in this direction is
known, however, in the context of crossed product $C^*$-algebras.
When the algebra is described as the crossed product $C^*$-algebra
$C(X) \rtimes_{\alpha} \mathbb{Z}$ of the algebra of continuous
functions on a compact Hausdorff space $X$ by an action of
$\mathbb{Z}$ via the composition automorphism associated with a
homeomorphism $\sigma : X \to X$, it is known that $C(X)$ sits inside
the $C^*$-crossed product as a maximal abelian subalgebra if and
only if for every positive integer $n$, the set of points in $X$
having period $n$ under iterations of $\sigma$ has no interior
points \cite[Theorem 5.4]{TomiyamaNotes1}, \cite[Corollary
3.3.3]{TomiyamaBook}, \cite[Theorems 2.8, 5.2]{TomiyamaNotes2},
\cite[Proposition 4.14]{Zeller-Meier}, \cite[Lemma
7.3.11]{LiBing-Ren}. This condition is equivalent to the action of
$\mathbb Z$ on $X$ being topologically free in the sense that the
non-periodic points of $\sigma$ are dense in $X$. In
\cite{svesildej}, a purely algebraic variant of the crossed product
allowing for more general classes of algebras than merely continuous
functions on compact Hausdorff spaces serving as ``base coefficient
algebras'' in the crossed products was considered. In the general
set theoretical framework of a crossed product algebra
$A \rtimes_{\alpha}\mathbb{Z}$ of an arbitrary subalgebra $A$ of
the algebra $\complex^X$ of complex-valued functions on a set $X$
(under the usual pointwise operations) by $\mathbb{Z}$ acting on $A$ via
a composition automorphism defined by a bijection of $X$, the
essence of the matter is revealed. Topological notions are not
available here and thus the condition of freeness of the dynamics as
described above is not applicable, so that it has to be generalized
in a proper way in order to be equivalent to the maximal
commutativity of $A$. In \cite{svesildej} such a generalization was
provided by involving separation properties of $A$ with respect to
the space $X$ and the action for significantly more arbitrary
classes of base coefficient algebras and associated spaces and
actions. The (unique) maximal abelian subalgebra containing $A$ was
described as well as general results and examples and
counterexamples on equivalence of maximal commutativity of $A$ in
the crossed product and the generalization of topological freeness
of the action.

In this article, we bring these results and interplay into a more
general algebraic context of crossed product rings (or algebras)
for crossed systems with arbitrary group actions and twisting
cocycle maps \cite{MoGR}. We investigate the connections with the
ideal structure of a general crossed product ring, describe
the center of crossed product rings,
describe the
commutant of the base coefficient subring in a crossed product
ring of a general crossed system, and obtain conditions for maximal commutativity of the commutant of the base
subring in terms of the action as well as in terms of intersection
of ideals in the crossed product ring with the base subring,
specially taking into account both the case of base rings without non-trivial
zero-divisors and the case of base rings with non-trivial zero-divisors.

\section{Preliminaries}
In this section, we recall the basic objects and the notation, from \cite{MoGR}, which are necessary for the understanding of the rest of this article.

\subsection*{Gradings}
Let $G$ be a group with unit element $e$. The ring $\rring$ is \emph{$G$-graded} if there is a family $\{\rring_\sigma\}_{\sigma \in G}$ of additive subgroups $\rring_\sigma$ of $\rring$ such that $\rring=\bigoplus_{\sigma \in G} \rring_\sigma$ and $\rring_\sigma \rring_\tau \subseteq \rring_{\sigma \tau}$ (\emph{strongly $G$-graded} if, in addition, $\supseteq$ also holds) for every $\sigma, \tau \in G$.

\subsection*{Crossed products}
If $\rring$ is a unital ring, then $U(\rring)$ denotes the group of multiplication invertible elements of $\rring$. A unital $G$-graded ring $\rring$ is called a \emph{$G$-crossed product} if $U(\rring)\cap \rring_\sigma \neq \emptyset$ for every $\sigma \in G$. Note that every $G$-crossed product is strongly $G$-graded, as explained in \cite[p.2]{MoGR}.

\subsection{Crossed systems}\label{crossedsystem}
\begin{defn}
A $G$-crossed system is a quadruple $\{\aalg,G,\sigma,\alpha\}$, consisting of a unital associative ring $\aalg$, a group $G$ (with unit element $e$), a map
\begin{displaymath}
\sigma : G \to \aut(\aalg)
\end{displaymath}
and a \emph{$\sigma$-cocycle} map
\begin{displaymath}
\alpha : G \times G \to U(\aalg)
\end{displaymath}
such that for any $x,y,z\in G$ and $a\in \aalg$ the following conditions hold:
\begin{itemize}
\item[(i)] $\sigma_x(\sigma_y(a)) = \alpha(x,y) \, \sigma_{xy}(a) \, \alpha(x,y)^{-1}$ \,;
\item[(ii)] $\alpha(x,y) \, \alpha(xy,z) = \sigma_x(\alpha(y,z)) \, \alpha(x,yz)$ \,;
\item[(iii)] $\alpha(x,e)=\alpha(e,x)=\unita$ \, .
\end{itemize}
\end{defn}

\begin{rem}\label{zeroandone}
Note that, by combining conditions (i) and (iii), we get $\sigma_e(\sigma_e(a))=\sigma_e(a)$ for all $a\in \aalg$. Furthermore, $\sigma_e : \aalg \to \aalg$ is an automorphism and hence $\sigma_e = \identity_{\aalg}$. Also note that, from the definition of $\aut(\aalg)$, we have $\sigma_g(\zeroa)=\zeroa$ and $\sigma_g(\unita)=\unita$ for any $g \in G$.
\end{rem}

\begin{rem}\label{conditionsAcomm}
From condition (i) it immediately follows that $\sigma$ is a group homomorphism if $\aalg$ is commutative or if $\alpha$ is trivial.
\end{rem}

\begin{defn}
Let $\overline{G}$ be a copy (as a set) of $G$. Given a $G$-crossed system $\{\aalg,G,\sigma,\alpha\}$, we denote by $\gcrossedring$ the free left $\aalg$-module having $\overline{G}$ as its basis and in which the multiplication is defined by
\begin{eqnarray}\label{leftmodulemultiplication}
(a_1 \overline{x})(a_2 \overline{y}) = a_1 \sigma_x(a_2) \, \alpha(x,y) \, \overline{xy}
\end{eqnarray}
for all $a_1,a_2 \in \aalg$ and $x,y \in G$. Elements of $\gcrossedring$ may be expressed as formal sums $\sum_{g\in G} a_g \overline{g}$ where $a_g \in \aalg$ and $a_g = \zeroa$ for all but a finite number of $g \in G$. 
Explicitly, this means that the addition and multiplication of two arbitrary elements $\sum_{s\in G} a_s \overline{s}, \sum_{t\in G} b_t \overline{t} \in \gcrossedring$ is given by
\begin{eqnarray}
\sum_{s\in G} a_s \overline{s} + \sum_{t\in G} b_t \overline{t} &=& \sum_{g\in G} (a_g + b_g) \overline{g} \nonumber \\
\left( \sum_{s\in G} a_s \overline{s} \right) \left( \sum_{t\in G} b_t \overline{t} \right) &=& \sum_{(s,t) \in G \times G} (a_s \overline{s})(b_t \overline{t}) =
\sum_{(s,t) \in G \times G} a_s \, \sigma_s(b_t) \, \alpha(s,t) \, \overline{st} \nonumber \\ &=& \sum_{g\in G} \left( \sum_{\{(s,t) \in G \times G \mid st=g\}} a_s \, \sigma_s(b_t) \, \alpha(s,t) \right) \overline{g} \label{product}.
\end{eqnarray}
\end{defn}

\begin{rem}
The ring $\aalg$ is unital, with unit element $\unita$, and it is easy to see that $(\unita \, \overline{e})$ is the multiplicative identity in $\gcrossedring$.
\end{rem}

By abuse of notation, we shall sometimes let $0$ denote the zero element in $\gcrossedring$. The proofs of the two following propositions can be found in \cite[Proposition 1.4.1, p.11]{MoGR} and \cite[Proposition 1.4.2, pp.12-13]{MoGR} respectively.

\begin{prop}
Let $\{\aalg,G,\sigma,\alpha\}$ be a $G$-crossed system. Then $\gcrossedring$ is an associative ring (with the multiplication defined in \eqref{leftmodulemultiplication}). Moreover, this ring is $G$-graded,  $\gcrossedring = \bigoplus_{g\in G} \, \aalg \overline{g}$, and it is a $G$-crossed product.
\end{prop}

\begin{prop}
Every $G$-crossed product $\rring$ is of the form $\gcrossedring$ for some ring $\aalg$ and some maps $\sigma,\alpha$.
\end{prop}

\begin{rem}
If $k$ is a field and $\aalg$ is a $k$-algebra, then so is $\gcrossedring$.
\end{rem}

\noindent The base coefficient ring $\aalg$ is naturally embedded as a subring into $\gcrossedring$. Consider the canonical isomorphism
\begin{displaymath}
\iota : \aalg \hookrightarrow \gcrossedring, \quad a \mapsto a \overline{e}.
\end{displaymath}
We denote by $\tilde{\aalg}$ the image of $\aalg$ under $\iota$ and by $\aalg^G = \{ a\in \aalg \, \mid \, \sigma_s(a)=a, \,\, \forall s\in G\}$ the \emph{fixed ring} of $\aalg$.

\begin{rem}\label{AcommAtildecomm}
Obviously, $\aalg$ is commutative if and only if $\tilde{\aalg}$ is commutative.
\end{rem}

\begin{exmp}
Let $\aalg$ be commutative and $\mathcal{B}=\gcrossedring$ a crossed product. For $x\in G$ and $c,d \in \aalg$ we may write
\begin{eqnarray*}
(c \, \overline{x})(d \, \overline{e}) = c \, \sigma_x(d) \, \overline{x} = (\sigma_x(d) \, \overline{e})(c \, \overline{x}).
\end{eqnarray*}
Let $b=c \,\overline{x}$, $a = d\overline{e}$ and $f : \mathcal{B} \to \mathcal{B}$ be a map defined by $f = \iota \circ \sigma_x \circ \iota^{-1}$. Then the above relation may be written as
\begin{displaymath}
b \, a = f(a) \, b
\end{displaymath}
which is a re-ordering formula frequently appearing in physical applications.
\end{exmp}

\section{Commutativity in $\gcrossedring$}

From the definition of the product in $\gcrossedring$, given by \eqref{product}, we see that two elements $\sum_{s\in G} a_s \overline{s}$ and $\sum_{t\in G} b_t \overline{t}$ commute if and only if
\begin{equation}\label{twoelementscommute}
\sum_{\{(s,t) \in G \times G \mid st=g\} } a_s \, \sigma_s(b_t) \, \alpha(s,t) = \sum_{\{(s,t) \in G \times G \mid st=g\}} b_s \, \sigma_s(a_t) \, \alpha(s,t)
\end{equation}
for each $g \in G$. The crossed product $\gcrossedring$ is in general non-commutative and in the following proposition we give a description of its center.

\begin{prop}\label{thecenter}
The center of $\gcrossedring$ is as follows
\begin{eqnarray*}
        Z(\gcrossedring) &=& \Big\{ \sum_{g\in G} r_g \, \overline{g} \,\, \Big\lvert \,\, r_{ts^{-1}} \, \alpha(ts^{-1},s) = \sigma_s(r_{s^{-1}t}) \, \alpha(s,s^{-1}t),\\
        & & \hspace{55pt} r_s \, \sigma_s(a) = a \, r_s, \quad \forall a\in \aalg, \,\, (s,t)\in G \times G \Big\}.
\end{eqnarray*}
\end{prop}

\begin{proof}
Let $\sum_{g\in G} r_g \, \overline{g} \in \gcrossedring$ be an element which commutes with every element of $\gcrossedring$. Then, in particular $\sum_{g\in G} r_g \, \overline{g}$ must commute with $a \, \overline{e}$ for every $a \in \aalg$. From \eqref{twoelementscommute} we immediately see that this implies $r_s \, \sigma_s(a)=a \, r_s$ for every $a\in \aalg$ and $s\in G$. Furthermore, $\sum_{g\in G} r_g \, \overline{g}$ must commute with $\unita \, \overline{s}$ for any $s \in G$. This yields
\begin{eqnarray*}
\sum_{t\in G} r_{ts^{-1}} \, \alpha(ts^{-1},s) \, \overline{t} = [gs= t] = \sum_{g\in G} r_g \, \alpha(g,s) \, \overline{gs}
= \sum_{g \in G} r_g \, \sigma_g(\unita) \, \alpha(g,s) \, \overline{gs}\\ = \left( \sum_{g \in G} r_g \, \overline{g} \right) (\unita \, \overline{s}) = (\unita \, \overline{s}) \left( \sum_{g \in G} r_g \, \overline{g} \right)
= \sum_{g \in G} \unita \, \sigma_s (r_g) \, \alpha(s,g) \, \overline{sg} \\ = \sum_{g \in G} \sigma_s (r_g) \, \alpha(s,g) \, \overline{sg} = [t=sg] = \sum_{t \in G} \sigma_{s} (r_{s^{-1}t}) \, \alpha(s,s^{-1}t) \, \overline{t}
\end{eqnarray*}
and hence, for each $(s,t) \in G \times G$, we have $r_{ts^{-1}}  \, \alpha(ts^{-1},s) = \sigma_s(r_{s^{-1}t}) \, \alpha(s,s^{-1}t)$.\\
Conversely, suppose that $\sum_{g\in G} r_g \, \overline{g} \in \gcrossedring$ is an element satisfying $r_s \, \sigma_s(a) = a \, r_s$ and $r_{ts^{-1}}  \, \alpha(ts^{-1},s) = \sigma_s(r_{s^{-1}t}) \, \alpha(s,s^{-1}t)$ for every $a\in \aalg$ and $(s,t)\in G\times G$. Let $\sum_{s \in G} a_s \, \overline{s} \in \gcrossedring$ be arbitrary. Then
\begin{eqnarray*}
\left( \sum_{g \in G} r_g \, \overline{g} \right) \left(\sum_{s \in G} a_s \, \overline{s} \right) = \sum_{(g,s) \in G \times G} r_g \, \sigma_g(a_s) \, \alpha(g,s) \, \overline{gs}
= \sum_{(g,s) \in G \times G} a_s \, r_g \, \alpha(g,s) \, \overline{gs} \\
= [gs=t] = \sum_{(t,s) \in G \times G} a_s \, (r_{ts^{-1}} \, \alpha(ts^{-1},s)) \, \overline{t}
= \sum_{(t,s) \in G \times G} a_s \, \sigma_s(r_{s^{-1}t}) \, \alpha(s,s^{-1}t) \, \overline{t} \\
= [t=sg] = \sum_{(g,s) \in G \times G} a_s \, \sigma_s(r_{g}) \, \alpha(s,g) \, \overline{sg}
= \left(\sum_{s \in G} a_s \, \overline{s} \right) \left( \sum_{g \in G} r_g \, \overline{g} \right)
\end{eqnarray*}
and hence $\sum_{g \in G} r_g \, \overline{g}$ commutes with every element of $\gcrossedring$.
\end{proof}

A few corollaries follow from Proposition \ref{thecenter}, showing how a successive addition of restrictions on the corresponding $G$-crossed system, leads to a simplified description of $Z(\gcrossedring)$.

\begin{cor}[Center of a twisted group ring]
Let $\sigma \equiv \ida$. Then, the center of $\gcrossedring$ is as follows
\begin{eqnarray*}
        Z(\gcrossedring) &=& \Big\{ \sum_{g\in G} r_g \, \overline{g} \,\, \Big\vert \,\, r_s \in Z(\aalg), \quad r_{ts^{-1}}  \, \alpha(ts^{-1},s) = r_{s^{-1}t} \, \alpha(s,s^{-1}t), \\
        & & \hspace{170pt} \forall a\in \aalg, \,\, (s,t)\in G \times G \Big\}.
\end{eqnarray*}
\end{cor}

\begin{cor}
Let $G$ be abelian and $\alpha$ symmetric\footnote{Symmetric in the sense that $\alpha(x,y)=\alpha(y,x)$ for every $(x,y) \in G \times G$.}. Then, the center of $\gcrossedring$ is as follows
\begin{eqnarray*}
        Z(\gcrossedring) = \Big\{ \sum_{g\in G} r_g \, \overline{g} \,\, \Big\lvert \,\, r_s \, \sigma_s(a) = a \, r_s, \quad r_s \in \aalg^G, \quad \forall a\in \aalg, \,\, s\in G \Big\}.
\end{eqnarray*}
\end{cor}

\begin{cor}\label{centerspecial}
Let $\aalg$ be commutative, $G$ abelian and $\alpha \equiv \unita$. Then, the center of $\gcrossedring$ is as follows
\begin{eqnarray*}
        Z(\gcrossedring) = \Big\{ \sum_{g\in G} r_g \, \overline{g} \,\, \Big\vert \,\, r_s \in \aalg^G, \,\, \sigma_s(a) - a \in \ann(r_s), \,\, \forall a\in \aalg, \, s\in G \Big\}.
\end{eqnarray*}
\end{cor}

\begin{rem}
Note that in the proof of Theorem \ref{thecenter}, the property that the image of $\alpha$ is contained in $U(\aalg)$ is not used and therefore the theorem is true in greater generality. Consider the case when $\aalg$ is an integral domain and let $\alpha$ take its values in $\aalg \setminus \{\zeroa\}$. In this case it is clear that $r_s \, \sigma_s(a) = a \, r_s$ for all $a\in \aalg$ $\Longleftrightarrow$ $r_s ( \sigma_s(a) - a) = 0$ for all $a\in \aalg$ $\Longleftrightarrow r_s = 0$ for $s \not \in \sigma^{-1}(\ida) = \{g \in G \mid \sigma_g = \ida \}$. After a change of variable via $x=s^{-1}t$ the first condition in the description of the center may be written as $\sigma_s(r_x) \, \alpha(s,x) = r_{sxs^{-1}} \, \alpha(sxs^{-1},s)$ for all $(s,x) \in G \times G$. From this relation we conclude that $r_x=0$ if and only if $r_{sxs^{-1}}=0$, and hence it is trivially satisfied if we put $r_x = 0$ whenever $x \not \in \sigma^{-1}(\ida)$. This case has been presented in \cite[Proposition 2.2]{NeOyYu} with a more elaborate proof.
\end{rem}

The final corollary describes the exceptional situation when $Z(\gcrossedring)$ coincides with $\gcrossedring$, that is when $\gcrossedring$ is commutative.

\begin{cor}
$\gcrossedring$ is commutative if and only if all of the following hold
\begin{itemize}
\item[(i)] $\aalg$ is commutative,
\item[(ii)] $\sigma_s = \ida$ for each $s\in G$,
\item[(iii)] $G$ is abelian,
\item[(iv)] $\alpha$ is symmetric.
\end{itemize}
\end{cor}

\begin{proof}
Suppose that $Z(\gcrossedring)=\gcrossedring$. Then, $\tilde{\aalg} \subseteq \gcrossedring = Z(\gcrossedring)$ and hence (i) follows by Remark \ref{AcommAtildecomm}. By the assumption, $\unita \, \overline{s} \in Z(\gcrossedring)$ for any $s\in G$ and by Proposition \ref{thecenter} we see that $\sigma_s = \ida$ for every $s\in G$, and hence (ii). For any $(x,y) \in G \times G$ we have
$\alpha(x,y) \, \overline{xy} = (\unita \overline{x})(\unita \overline{y}) = (\unita \overline{y})(\unita \overline{x}) = \alpha(y,x) \, \overline{yx}$, but $\alpha(x,y) \neq \zeroa$ which implies $xy=yx$ and also $\alpha(x,y)=\alpha(y,x)$, which shows (iii) and (iv). The converse statement of the corollary is easily verified.
\end{proof}

\section{The commutant of $\tilde{\aalg}$ in $\gcrossedring$}

From now on we shall assume that $G \neq \{e\}$. As we have seen, $\tilde{\aalg}$ is a subring of $\gcrossedring$ and we define its commutant by
\begin{displaymath}
\comm(\tilde{\aalg}) = \{b \in \gcrossedring \, \mid \, ab = ba, \quad \forall a \in \tilde{\aalg} \}.
\end{displaymath}

\noindent Theorem \ref{commutantcomm} tells us exactly when an element of $\gcrossedring$ lies in $\comm(\tilde{\aalg})$.

\begin{thm}\label{commutantcomm}
The commutant of $\tilde{\aalg}$ in $\gcrossedring$ is as follows
    \begin{eqnarray*}
    \comm(\tilde{\aalg}) = \left\{ \sum_{s \in G} r_s \overline{s} \in \gcrossedring \,\, \Big\lvert \,\, r_s \, \sigma_s(a) = a \, r_s,    \quad \forall a \in \aalg , \, s\in G \right\}.
    \end{eqnarray*}
\end{thm}

\begin{proof}
The proof is established through the following sequence of equivalences:
    \begin{eqnarray*}
        \sum_{s \in G} r_s \overline{s} \in \comm(\tilde{\aalg}) \Longleftrightarrow \left( \sum_{s \in G} r_s \overline{s} \right) \left( a \overline{e} \right) = \left( a \overline{e} \right) \left( \sum_{s \in G} r_s \overline{s} \right), \quad \forall a \in \aalg \\
        \Longleftrightarrow \sum_{s \in G} r_s \, \sigma_s(a) \, \alpha(s,e) \, \overline{se} = \sum_{s \in G} a \, \sigma_e(r_s) \, \alpha(e,s) \, \overline{es}, \quad \forall a \in \aalg \\
        \Longleftrightarrow \sum_{s \in G} r_s \, \sigma_s(a) \, \overline{s} = \sum_{s \in G} a \, r_s \, \overline{s}, \quad \forall a \in \aalg \\
\Longleftrightarrow \text{For each } s \in G: \, \, r_s \, \sigma_s(a) = a \, r_s, \quad \forall a \in \aalg
    \end{eqnarray*}
Here we have used the fact that $\alpha(s,e)=\alpha(e,s)=\unita$ for all $s\in G$. The above equivalence can also be deduced directly from \eqref{twoelementscommute}.
\end{proof}

In the case when $\aalg$ is commutative we get the following description of the commutant as a special case of Theorem \ref{commutantcomm}.

\begin{cor}\label{commutantcomm2}
If $\aalg$ is commutative, then the commutant of $\tilde{\aalg}$ in $\gcrossedring$ is
    \begin{eqnarray*}
    \comm(\tilde{\aalg}) = \left\{ \sum_{s \in G} r_s \overline{s} \in \gcrossedring \,\, \Big\lvert \,\, \sigma_s(a)-a \in \ann(r_s),  \quad \forall a \in \aalg , \, s\in G \right\}.
    \end{eqnarray*}
\end{cor}

When $\aalg$ is commutative it is clear that $\tilde{\aalg} \subseteq \comm(\tilde{\aalg})$. Using the explicit description of $\comm(\tilde{\aalg})$ in Corollary \ref{commutantcomm2}, we are now able to state exactly when $\tilde{\aalg}$ is maximal commutative, i.e. $\comm(\tilde{\aalg})=\tilde{\aalg}$.

\begin{cor}\label{maxabelianequivstatement}
Let $\aalg$ be commutative. Then $\tilde{\aalg}$ is maximal commutative in $\gcrossedring$ if and only if, for each pair $(s,r_s) \in (G \setminus\{e\}) \times (\aalg \setminus \{\zeroa\})$, there exists $a\in \aalg$ such that $\sigma_s(a)-a \not \in \ann(r_s) = \{ c \in \aalg \, \mid \, r_s \cdot c = \zeroa \}$.
\end{cor}

\begin{exmp}\label{introtillchristians}
In this example we follow the notation of \cite{svesildej}. Let $\sigma : X \to X$ be a bijection
on a non-empty set $X$, and $A \subseteq \complex^X$ an algebra of functions, such that
if $h \in A$ then $h \circ \sigma \in A$ and $h \circ \sigma^{-1} \in A$. Let $\tilde{\sigma} : \integers \to \aut(A)$ be defined by $\tilde{\sigma}_n : f \mapsto f \circ \sigma^{\circ (-n)}$ for $f\in A$. We now have a
$\integers$-crossed system (with trivial $\tilde{\sigma}$-cocycle) and we may form the crossed product
$A \rtimes_{\tilde{\sigma}} \integers$. Recall the definition of the set $\sep_A^n(X) = \{x \in X \, \mid \, \exists h \in A, \text{ s.t. }
 h(x) \neq (\tilde{\sigma}_n(h))(x) \}$. Corollary \ref{maxabelianequivstatement} is a generalization of \cite[Theorem 3.5]{svesildej} and the easiest way to see this is by negating the statements. Suppose that $A$ is not maximal commutative in $A \rtimes_{\tilde{\sigma}} \integers$. Then, by Corollary \ref{maxabelianequivstatement}, there exists a pair $(n,f_n) \in (\integers \setminus \{0\}) \times (A \setminus \{0\})$ such that $\tilde{\sigma}_n(g)-g \in \ann(f_n)$ for every $g \in A$, i.e. $\supp(\tilde{\sigma}_n(g)-g) \cap \supp(f_n) = \emptyset$ for every $g \in A$. In particular, this means that $f_n$ is identically zero on $\sep_A^n(X)$. However, $f_n \in A \setminus \{0\}$ is not identically zero on $X$ and hence $\sep_A^n(X)$ is not a \emph{domain of uniqueness} (as defined in \cite[Definition 3.2]{svesildej}). The converse can be proved similarly.
\end{exmp}

\begin{cor}\label{commutantzerodiv}
Let $\aalg$ be commutative. If for each $s \in G \setminus \{e\}$ it is always possible to find some $a \in \aalg$ such that $\sigma_s(a)-a$ is not a zero-divisor in $\aalg$, then $\tilde{\aalg}$ is maximal commutative in $\gcrossedring$.
\end{cor}

The next corollary is a consequence of Corollary \ref{maxabelianequivstatement} and shows how maximal commutativity of the base coefficient ring in the crossed product has an impact on the non-triviality of the action $\sigma$.

\begin{cor}\label{maxabeliansigma}
Let the subring $\tilde{\aalg}$ be maximal commutative in $\gcrossedring$, then $\sigma_g \neq \ida$ for all $g\in G \setminus \{e\}$.
\end{cor}

The description of the commutant $\comm(\tilde{\aalg})$ from Corollary \ref{commutantcomm2} can be further refined in the case when $\aalg$ is an integral domain.

\begin{cor}\label{commutantcomm3}
If $\aalg$ is an integral domain\footnote{By an \emph{integral domain} we shall mean a commutative ring with an additive identity $\zeroa$ and a multiplicative identity $\unita$ such that $\zeroa \neq \unita$, in which the product of any two non-zero elements is always non-zero.}, then the commutant of $\tilde{\aalg}$ in $\gcrossedring$ is
    \begin{eqnarray*}
    \comm(\tilde{\aalg}) = \Big\{ \sum_{s\in \sigma^{-1}(\ida) } r_s \overline{s} \in \gcrossedring \,\, \Big\lvert \,\, r_s \in \aalg \Big\}
    \end{eqnarray*}
    where $\sigma^{-1}(\ida) = \{g \in G \mid \sigma_g = \ida \}$.
\end{cor}

The following corollary can be derived directly from Corollary \ref{maxabeliansigma} together with either Corollary \ref{commutantzerodiv} or Corollary \ref{commutantcomm3}.

\begin{cor}\label{maxabeliansigmaconverse}
Let $\aalg$ be an integral domain. Then $\sigma_g \neq \ida$ for all $g\in G \setminus \{e\}$ if and only if $\tilde{\aalg}$ is maximal commutative in $\gcrossedring$.
\end{cor}

\begin{rem}\label{injectivityofsigma}
Recall that when $\aalg$ is commutative, $\sigma$ is a group homomorphism. Thus, to say that $\sigma_g \neq \ida$ for all $g\in G\setminus\{e\}$ is just another way of saying that $\ker(\sigma)=\{e\}$ or equivalently, that $\sigma$ is injective.
\end{rem}

\begin{exmp}
Let $\aalg =\complex[x_1,\ldots,x_n]$ be the polynomial ring in $n$ commuting variables $x_1,\ldots,x_n$ and $G = S_n$ the symmetric group of $n$ elements. An element $\tau \in S_n$ is a permutation which maps the sequence $(1,\ldots,n)$ into $(\tau(1),\ldots,\tau(n))$. The group $S_n$ acts on $\complex[x_1,\ldots,x_n]$ in a natural way.
To each $\tau \in S_n$ we may associate a map $\aalg \to \aalg$, which sends
any polynomial $f(x_1,...,x_n) \in \complex[x_1,\ldots,x_n]$ into a new polynomial $g$, defined by $g(x_1,\ldots,x_n) = f(x_{\tau(1)},\ldots,x_{\tau(n)})$. It is clear that each such mapping is a ring automorphism on $\aalg$. Let $\sigma$ be the embedding $S_n \hookrightarrow \aut(\aalg)$ and $\alpha \equiv \unita$. Note that $\complex[x_1,\ldots,x_n]$ is an integral domain and that $\sigma$ is injective. Hence, by Corollary \ref{maxabeliansigmaconverse} and Remark \ref{injectivityofsigma} it is clear that the embedding of $\complex[x_1,\ldots,x_n]$ is maximal commutative in $\complex[x_1,\ldots,x_n] \rtimes^{\sigma} S_n$.
\end{exmp}

One might want to describe properties of the $\sigma$-cocycle in the case when $\tilde{\aalg}$ is maximal commutative, but unfortunately this will lead to a dead end. The explaination for this is revealed by condition (iii) in the definition of a $G$-crossed system, where we see that $\alpha(e,g)=\alpha(g,e)=\unita$ for all $g \in G$ and hence we are not able to extract any interesting information about $\alpha$ by assuming that $\tilde{\aalg}$ is maximal commutative. At the end of \cite[Remark 1.4.3, part 2]{MoGR} it is explained that in a \emph{twisted group ring} $\aalg \rtimes_\alpha G$, i.e. with $\sigma \equiv \ida$, $\tilde{\aalg}$ can never be maximal commutative (for $G \neq \{e\}$). When $\aalg$ is commutative, this follows immediately from Corollary \ref{maxabeliansigma}.

We will now give a sufficient condition for $\comm(\tilde{\aalg})$ to be commutative.

\begin{prop}\label{commcomm}
Let $\aalg$ be a commutive ring, $G$ an abelian group and $\alpha$ symmetric. Then $\comm(\tilde{\aalg})$ is commutative.
\end{prop}

\begin{proof}
Let $\sum_{s \in G} r_s \overline{s}$ and $\sum_{t \in G} p_t \overline{t}$ be arbitrary elements of $\comm(\tilde{\aalg})$, then by our assumptions and Corollary \ref{commutantcomm2} we get
\begin{eqnarray*}
\left(\sum_{s \in G} r_s \, \overline{s} \right) \left( \sum_{t \in G} p_t \, \overline{t} \right) &=& \sum_{(s,t) \in G \times G} r_s \, \sigma_s(p_t) \, \alpha(s,t) \, \overline{st}
= \sum_{(s,t) \in G \times G} r_s \, p_t \, \alpha(s,t) \, \overline{st} \\
&=& \sum_{(s,t) \in G \times G} p_t \, \sigma_t(r_s) \, \alpha(t,s) \, \overline{ts} = \left( \sum_{t \in G} p_t \, \overline{t} \right) \left(\sum_{s \in G} r_s \, \overline{s} \right).
\end{eqnarray*}
This shows that $\comm(\tilde{\aalg})$ is commutative.
\end{proof}

This proposition is a generalization of \cite[Proposition 2.1]{svesildej} from a function algebra to an arbitrary unital associative commutative ring $\aalg$, from $\integers$ to an arbitrary abelian group $G$ and from a trivial to a possibly non-trivial symmetric $\sigma$-cocycle $\alpha$.

\begin{rem}
By using Proposition \ref{commcomm} and the arguments made in Example \ref{introtillchristians} it is clear that Corollary \ref{commutantcomm2} is a generalization of \cite[Theorem 3.3]{svesildej}. Furthermore, we see that Corollary \ref{centerspecial} is a generalization of \cite[Theorem 3.6]{svesildej}.
\end{rem}

\section{Ideals in $\gcrossedring$}

In this section we describe properties of the ideals in $\gcrossedring$ in connection with maximal commutativity and properties of the action $\sigma$.


\begin{thm}\label{killerproof}
Let $\aalg$ be commutative. Then
\begin{displaymath}
I \cap \comm(\tilde{\aalg}) \neq \{0\}
\end{displaymath}
for every non-zero two-sided ideal $I$ in $\gcrossedring$.
\end{thm}

\begin{proof}
Let $\aalg$ be commutative. Then $\tilde{\aalg}$ is also commutative. Let $I \subseteq \gcrossedring$ be an arbitrary non-zero two-sided ideal in $\gcrossedring$. \\

\noindent \emph{Part 1:} \\
For each $g \in G$ we may define a \emph{translation-deformation operator}
\begin{displaymath}
T_g : \gcrossedring \to \gcrossedring, \quad \sum_{s \in G} a_s \overline{s} \mapsto \left(\sum_{s \in G} a_s \overline{s}\right) (\unita \overline{g}).
\end{displaymath}
Note that, for any $g \in G$, $I$ is invariant\footnote{By \emph{invariant} we mean that the set is \emph{closed} under this operation.} under $T_g$. We have
\begin{displaymath}
T_g \left( \sum_{s \in G} a_s \overline{s} \right) = \left(\sum_{s \in G} a_s \overline{s}\right) (\unita \overline{g}) = \sum_{s \in G} a_s \, \sigma_s(\unita) \, \alpha(s,g) \overline{sg} = \sum_{s \in G} a_s \, \alpha(s,g) \overline{sg}
\end{displaymath}
for every $g\in G$. It is important to note that if $a_s \neq \zeroa$, then $a_s \, \alpha(s,g) \neq \zeroa$ and hence this operation does not kill coefficients, it only translates and deformes them. If we have a non-zero element $\sum_{s \in G} a_s \overline{s}$ for which $a_e =\zeroa$, then we may pick some non-zero coefficient, say $a_p$ and apply the operator $T_{p^{-1}}$ to end up with
\begin{displaymath}
T_{p^{-1}} \left( \sum_{s \in G} a_s \overline{s} \right) = \sum_{s \in G} a_s \, \alpha(s,p^{-1}) \overline{sp^{-1}} := \sum_{s \in G} d_t \overline{t}.
\end{displaymath}
This resulting element will then have the following properties:
\begin{itemize}
\item $d_e = a_p \, \alpha(p,p^{-1}) \neq \zeroa$,
\item $\# \{ s\in G \, \mid \, a_s \neq \zeroa \} = \# \{ s\in G \, \mid \, d_s \neq \zeroa \}$.
\end{itemize}
\vspace{10pt}
\noindent \emph{Part 2:} \\
Next we define a \emph{kill operator}
\begin{displaymath}
D_a : \gcrossedring \to \gcrossedring, \quad \sum_{s \in G} a_s \overline{s} \mapsto (a \overline{e}) \left( \sum_{s \in G} a_s \overline{s} \right) - \left( \sum_{s \in G} a_s \overline{s} \right) (a \overline{e})
\end{displaymath}
for each $a \in \aalg$. Note that, for each $a \in \aalg$, $I$ is invariant under $D_a$. By assumption $\aalg$ is commutative and hence the above expression can be simplified
\begin{eqnarray*}
D_a \left( \sum_{s \in G} a_s \overline{s} \right) &=& (a \overline{e}) \left( \sum_{s \in G} a_s \overline{s} \right) - \left( \sum_{s \in G} a_s \overline{s} \right) (a \overline{e}) \\
&=& \left( \sum_{s \in G} a \, \sigma_{e}(a_s) \, \alpha(e,s) \, \overline{es} \right) - \left( \sum_{s \in G} a_s \, \sigma_{s}(a) \, \alpha(s,e) \, \overline{se} \right) \\
&=& \sum_{s \in G} \underbrace{a \, a_s}_{=a_s \, a} \overline{s} - \sum_{s \in G} a_s \, \sigma_{s}(a) \, \overline{s} = \sum_{s \in G} a_s \, (a-\sigma_{s}(a)) \overline{s} \\
&=& \sum_{s \neq e} a_s \, (a-\sigma_{s}(a)) \overline{s} = \sum_{s \neq e} d_s \overline{s}.
\end{eqnarray*}
The operators $\{D_a\}_{a \in \aalg}$ all share the property that they kill the coefficient in front $\overline{e}$. Hence, if $a_e \neq \zeroa$, then the number of non-zero coefficients of the resulting element will always be reduced by at least one. Note that 
$\comm(\tilde{\aalg}) = \bigcap_{a \in \aalg} \ker(D_a)$. This means that for each non-zero $\sum_{s \in G} a_s \overline{s}$ in $\gcrossedring \setminus \comm(\tilde{\aalg})$ we may always choose some $a \in A$ such that $\sum_{s \in G} a_s \overline{s} \not \in \ker(D_a)$. By choosing such an $a$ we note that, using the same notation as above, we get
\begin{displaymath}
\# \{s \in G \, \mid \, a_s \neq \zeroa \} \geq \# \{s \in G \, \mid \, d_s \neq \zeroa \} \geq 1
\end{displaymath}
for each non-zero $\sum_{s \in G} a_s \overline{s} \in \gcrossedring \setminus \comm(\tilde{\aalg})$. \\ \\
\noindent \emph{Part 3:} \\
The ideal $I$ is assumed to be non-zero, which means that we can pick some non-zero element $\sum_{s\in G} r_s \overline{s} \in I$. If $\sum_{s\in G} r_s \overline{s} \in \comm(\tilde{\aalg})$, then we are finished, so assume that this is not the case. Note that $r_s \neq \zeroa$ for finitely many $s \in G$. Recall that the ideal $I$ is invariant under $T_g$ and $D_a$ for all $g \in G$ and $a\in \aalg$. We may now use the operators $\{T_g\}_{g\in G}$ and $\{D_a\}_{a\in \aalg}$ to generate new elements of $I$. More specifically, we may use the $T_g$:s to translate our element $\sum_{s\in G} r_s \overline{s}$ into a new element which has a non-zero coefficient in front of $\overline{e}$ (if needed) after which we use the $D_a$ operator to kill this coefficient and end up with yet another new element of $I$ which is non-zero but has a smaller number of non-zero coefficients. We may repeat this procedure and in a finite number of iterations arrive at an element of $I$ which lies in $\comm(\tilde{\aalg}) \setminus \tilde{\aalg}$ and if not we continue the above procedure until we reach an element which is of the form $b \, \overline{e}$ with some non-zero $b \in \aalg$. In particular $\tilde{\aalg} \subseteq \comm(\tilde{\aalg})$ and hence $I \cap \comm(\tilde{\aalg}) \neq \{0\}$.
\end{proof}

The embedded base ring $\tilde{\aalg}$ is maximal commutative if and only if $\tilde{\aalg} = \comm(\tilde{\aalg})$ and hence we have the following corollary.

\begin{cor}\label{maxcommcorollary}
Let the subring $\tilde{\aalg}$ be maximal commutative in $\gcrossedring$. Then
\begin{displaymath}
I \cap \tilde{\aalg} \neq \{0\}
\end{displaymath}
for every non-zero two-sided ideal $I$ in $\gcrossedring$.
\end{cor}


\begin{prop}\label{leftidealprop}
Let $I$ be a subset of $\aalg$ and define
\begin{displaymath}
J = \left\{ \sum_{s\in G} a_s \, \overline{s} \in \gcrossedring \, \mid \, a_s \in I \right\}.
\end{displaymath}
Then the following assertions hold:
\begin{itemize}
\item[(i)] If $I$ is a right ideal in $\aalg$, then $J$ is a right ideal in $\gcrossedring$.
\item[(ii)] If $I$ is a two-sided ideal in $\aalg$ such that $I \subseteq \aalg^G$, then $J$ is a two-sided ideal in $\gcrossedring$. 
\end{itemize}
\end{prop}

\begin{proof}
If $I$ is a (possibly one sided) ideal in $\aalg$, it is clear that $J$ is an additive subgroup of $\gcrossedring$.\\
(i). Let $I$ be a right ideal in $\aalg$. Then
\begin{eqnarray*}
\left( \sum_{s\in G} a_s \, \overline{s} \right) \left( \sum_{t \in G} b_t \, \overline{t} \right)  &=& \sum_{(s,t) \in G \times G} \underbrace{ a_s \, \sigma_s(b_t) \, \alpha(s,t) }_{\in I} \overline{st} \in J
\end{eqnarray*}
for arbitrary $\sum_{s\in G} a_s \, \overline{s} \in J$ and $\sum_{t \in G} b_t \, \overline{t} \in \gcrossedring$ and hence $J$ is a right ideal. \\
(ii). Let $I$ be a two-sided ideal in $\aalg$ such that $I \subseteq \aalg^G$. By (i) it is clear that $J$ is a right ideal. Let $\sum_{s\in G} a_s \, \overline{s} \in J$ and $\sum_{t \in G} b_t \, \overline{t} \in \gcrossedring$ be arbitrary. Then
\begin{eqnarray*}
\left( \sum_{t \in G} b_t \, \overline{t} \right) \left( \sum_{s\in G} a_s \, \overline{s} \right) = \sum_{(t,s) \in G \times G} b_t \, \sigma_t(a_s) \, \alpha(t,s) \overline{ts} = \sum_{(t,s) \in G \times G} \underbrace{ b_t \, a_s \, \alpha(t,s) }_{\in I} \overline{ts} \in J
\end{eqnarray*}
which shows that $J$ is also a left ideal.
\end{proof}

\begin{thm}\label{explicitideal}
Let $\sigma : G \to \aut(\aalg)$ be a group homomorphism and $N$ be a normal subgroup of $G$, contained in $\sigma^{-1}(\ida) = \{g\in G \mid \sigma_g = \ida\}$. Let $\varphi : G \to G/N$ be the quotient group homomorphism and suppose that $\alpha$ is such that $\alpha(s,t)=\unita$ whenever $s \in N$ or $t \in N$. Furthermore, suppose that there exists a map $\beta : G/N \times G/N \to U(\aalg)$ such that $\beta(\varphi(s),\varphi(t))= \alpha(s,t)$ for each $(s,t) \in G \times G$. Let $I$ be the ideal in $\gcrossedring$ generated by an element $\sum_{s\in N} a_s \, \overline{s}$ for which the coefficients (of which all but finitely many are zero) satisfy $\sum_{s\in N} a_s = \zeroa$. Then,
\begin{displaymath}
I \cap \tilde{\aalg} = \{0\}.
\end{displaymath}
\end{thm}

\begin{proof}
Let $I \subseteq \gcrossedring$ be the ideal generated by an element $\sum_{s\in N} a_s \, \overline{s}$, which satisfies $\sum_{s\in N} a_s = \zeroa$. The quotient homomorphism $\varphi : G \to G/N, \quad s \mapsto sN$ satisfies $\ker(\varphi)=N$. By assumption, the map $\sigma$ is a group homomorphism and $\sigma(N)= \ida$. Hence by the universal property, see for example \cite[p.16]{Lang}, there exists a unique group homomorphism $\rho$ making the following diagram commute:

\[
\xymatrix{
G \ar[r]^-\sigma \ar[d]_\varphi & \aut(\aalg) \\
G/N \ar@{-->}[ur]_\rho}
\]


\noindent By assumption there exists $\beta$ such that $\beta(\varphi(s),\varphi(t))= \alpha(s,t)$ for each $(s,t) \in G \times G$. One may verify that $\beta$ is a $\rho$-cocycle and hence we can define a new crossed product $\aalg \rtimes_\beta^{\rho} G/N$. We now define $\Gamma$ to be the map
\begin{displaymath}
\Gamma : \gcrossedring \to \aalg \rtimes_\beta^{\rho} G/N, \quad \sum_{s \in G} a_s \overline{s} \mapsto \sum_{s \in G} a_s \overline{\varphi(s)}
\end{displaymath}
and show that it is a ring homomorphism. For any two elements $\sum_{s \in G} a_s \overline{s}$ and $\sum_{t \in G} b_t \overline{t}$ in $\gcrossedring$, the additivity of $\Gamma$ follows by
\begin{eqnarray*}
\Gamma \left(\sum_{s \in G} a_s \overline{s} + \sum_{t \in G} b_t \overline{t} \right) &=& \Gamma \left(\sum_{s \in G} (a_s + b_s) \overline{s} \right) = \sum_{s \in G} (a_s + b_s) \overline{\varphi(s)} \\
&=& \sum_{s \in G} a_s \overline{\varphi(s)} + \sum_{t \in G} b_t \overline{\varphi(t)} = \Gamma \left(\sum_{s \in G} a_s \overline{s} \right) + \Gamma \left(\sum_{t \in G} b_t \overline{t} \right)
\end{eqnarray*}
and due to the assumptions, the multiplicativity follows by
\begin{eqnarray*}
\Gamma \left(\sum_{s \in G} a_s \overline{s} \sum_{t \in G} b_t \overline{t} \right) &=&
\Gamma \left(\sum_{(s,t) \in G \times G} a_s \sigma_s(b_t) \, \alpha(s,t) \, \overline{st} \right) \\
&=& \sum_{(s,t) \in G \times G} \, a_s \, \sigma_s(b_t) \, \alpha(s,t) \, \overline{\varphi(st)} \\
&=& \sum_{(s,t) \in G \times G} \, a_s \, \rho_{\varphi(s)}(b_t) \, \alpha(s,t) \, \overline{\varphi(s) \varphi(t)} \\
&=& \sum_{(s,t) \in G \times G} \, a_s \, \rho_{\varphi(s)}(b_t) \, \beta(\varphi(s),\varphi(t))\, \overline{\varphi(s) \varphi(t)} \\
&=& \left(\sum_{s \in G} a_s \overline{\varphi(s)} \right) \left(\sum_{t \in G} b_t \overline{\varphi(t)} \right) = \Gamma \left(\sum_{s \in G} a_s \overline{s} \right) \, \Gamma \left(\sum_{t \in G} b_t \overline{t} \right) \\
\end{eqnarray*}
and hence $\Gamma$ defines a ring homomorphism. We shall note that the generator of $I$ is mapped onto zero, i.e.
\begin{displaymath}
\Gamma \left( \sum_{s \in N} a_s \, \overline{s} \right)
= \sum_{s \in N} a_s \, \overline{\varphi(s)} = \sum_{s \in N} a_s \, \overline{N} =
\left( \sum_{s \in N} a_s \right) \, \overline{N} = \zeroa \, \overline{N} = 0
\end{displaymath}
and hence $\Gamma \lvert_{I} \equiv 0$. Furthermore, we see that
\begin{displaymath}
\Gamma \left(b \overline{e} \right) = 0 \quad \Longrightarrow \quad b \overline{\varphi(e)}=0 \Leftrightarrow b \overline{N}=0 \Leftrightarrow b = \zeroa \Leftrightarrow b\overline{e} = 0
\end{displaymath}
and hence $\Gamma\lvert_{\tilde{\aalg}}$ is injective. We may now conclude that if $c \in I \cap \tilde{\aalg}$, then $\Gamma(c)=0$ and so necessarily $c=0$. This shows that $I \cap \tilde{\aalg} = \{0\}$.
\end{proof}

When $\aalg$ is commutaive, $\sigma$ is automatically a group homomorphism and we get the following corollary.

\begin{cor}\label{explicitideal2}
Let $\aalg$ be commutative and $N \subseteq \sigma^{-1}(\ida) = \{g\in G \mid \sigma_g = \ida\}$ a normal subgroup of $G$. Let $\varphi : G \to G/N$ be the quotient group homomorphism and suppose that $\alpha$ is such that $\alpha(s,t)=\unita$ whenever $s \in N$ or $t \in N$. Furthermore, suppose that there exists a map $\beta : G/N \times G/N \to U(\aalg)$ such that $\beta(\varphi(s),\varphi(t))= \alpha(s,t)$ for each $(s,t) \in G \times G$. Let $I$ be the ideal in $\gcrossedring$ generated by an element $\sum_{s\in N} a_s \, \overline{s}$ for which the coefficients (of which all but finitely many are zero) satisfy $\sum_{s\in N} a_s = \zeroa$. Then,
\begin{displaymath}
I \cap \tilde{\aalg} = \{0\}.
\end{displaymath}
\end{cor}

When $\alpha \equiv \unita$ there is no need to assume that $\aalg$ is commutative, in order to make $\sigma$ a group homomorphism. In this case we may choose $\beta \equiv \unita$. Thus, by Theorem \ref{explicitideal} we have the following corollaries.

\begin{cor}\label{explicitidealtrivial}
Let $\alpha \equiv \unita$ and $N \subseteq \sigma^{-1}(\ida) = \{g\in G \mid \sigma_g = \ida\}$ be a normal subgroup of $G$. Let $I$ be the ideal in $\gcrossedring$ generated by an element $\sum_{s\in N} a_s \, \overline{s}$ for which the coefficients (of which all but finitely many are zero) satisfy $\sum_{s\in N} a_s = \zeroa$. Then,
\begin{displaymath}
I \cap \tilde{\aalg} = \{0\}.
\end{displaymath}
\end{cor}


\begin{cor}\label{nontrivialintersectionimplytrivialsigma}
Let $\alpha \equiv \unita$. Then the following implication holds:
\begin{center}
{\rm(i)} $Z(G) \cap \sigma^{-1}(\ida) \neq \{e\}$. \\
$\Downarrow$ \\
{\rm (ii)} For each $g\in Z(G) \cap \sigma^{-1}(\ida)$, the ideal $I_g$ generated by the element $\sum_{n\in \integers} a_n \, \overline{g^n}$ for which $\sum_{n \in \integers} a_n =\zeroa$ has the property
$I_g \cap \tilde{\aalg} = \{0\}$.
\end{center}
\end{cor}

\begin{proof}
Suppose that there exists a $g \in (Z(G) \cap \sigma^{-1}(\ida)) \setminus\{e\}$. Let $I_g \subseteq \gcrossedring$ be the ideal generated by $\sum_{n\in \integers} a_n \, \overline{g^n}$, where $\sum_{n \in \integers} a_n =\zeroa$. The element $g$ commutes with each element of $G$ and hence the cyclic subgroup $N = \langle g \rangle$ generated by $g$ is normal in $G$ and since $\sigma$ is a group homomorphism $N \subseteq \sigma^{-1}(\ida)$. Hence $I_g \cap \tilde{\aalg} = \{0\}$ by Corollary \ref{explicitidealtrivial}.
\end{proof}

\begin{cor}\label{idealtoaction}
Let $\alpha \equiv \unita$ and $G$ be abelian. Then the following implication holds:
\begin{center}
{\rm (i)} $I \cap \tilde{\aalg} \neq \{0\}$, for all non-zero two-sided ideals $I$ in $\gcrossedring$. \\
$\Downarrow$ \\
{\rm(ii)} $\sigma_g \neq \ida$ for all $g \in G \setminus \{e\}$.
\end{center}
\end{cor}

\begin{proof}[Proof by contrapositivity]

Since $G$ is abelian, $G=Z(G)$. Suppose that (ii) is false, i.e. there exists $g \in G \setminus \{e\}$ such that $\sigma_g = \ida$. Pick such a $g$ and let $I_g \subseteq \gcrossedring$ be the ideal generated by $\unita \, \overline{e} - \unita \, \overline{g}$. Then obviously $I_g \neq \{0\}$ and by Corollary \ref{nontrivialintersectionimplytrivialsigma} we get $I_g \cap \tilde{\aalg} = \{0\}$ and hence (i) is false. This concludes the proof.
\end{proof}

\begin{exmp}
We should note that in the proof of Corollary \ref{idealtoaction} one could have chosen the ideal in many different ways. The ideal generated by $\unita \overline{e} - \unita \overline{g} + \unita \overline{g^2} - \unita \overline{g^3} + \ldots + \unita \overline{g^{2n}} - \unita \overline{g^{2n+1}} = (\unita \overline{e} - \unita \overline{g}) \sum_{k=0}^{n} \unita \, g^{2k}$ is contained in the ideal $I_g$, generated by $\unita \overline{e} - \unita \overline{g}$, and therefore it has zero intersection with $\tilde{\aalg}$ if $I_g \cap \tilde{\aalg} = \{0\}$. Also note that for $\alpha \equiv \unita$ we may always write
\begin{displaymath}
\unita \overline{e} - \unita \overline{g^n} = (\unita \overline{e} - \unita \overline{g})(\sum_{k=0}^{n-1} \unita \, g^k)
\end{displaymath}
and hence $\unita \overline{e} - \unita \overline{g}$ is a zero-divisor in the crossed product $\gcrossedring$ whenever $g$ is a torsion element.
\end{exmp}

\begin{exmp}
We now give an example of how one may choose $\beta$ as in Theorem \ref{explicitideal}. Let $N \subseteq \sigma^{-1}(\ida)$ be a normal subgroup of $G$ such that for $g\in N$, $\alpha(s,g)=\unita$ for all $s\in G$ and let $\alpha$ be symmetric. Since $\alpha$ is the $\sigma$-cocycle map of a $G$-system, we get
\begin{eqnarray*}
\alpha(g,s) \, \alpha(gs,t) = \sigma_g(\alpha(s,t)) \, \alpha(g,st) \Longleftrightarrow
\alpha(g,s) \, \alpha(gs,t) = \alpha(s,t) \, \alpha(g,st) \\
\Longleftrightarrow \alpha(gs,t) = \alpha(s,t)
\end{eqnarray*}
for all $(s,t) \in G \times G$. Using the last equality and the symmetry of $\alpha$ we immediately see that
\begin{displaymath}
\alpha(g s, h t)= \alpha(s,t) \quad \forall s,t \in G
\end{displaymath}
for all $g,h \in N$. The last equality means that $\alpha$ is constant on the pairs of right cosets which coincide with the left cosets by normality of $N$. It is therefore clear that we can define $\beta : G/N \times G/N \to \aut(\aalg)$ by $\beta(\varphi(s),\varphi(t)) := \alpha(s,t)$ for any $s,t \in G$.
\end{exmp}

\begin{thm}\label{idealtomaxcomm}
Let $\aalg$ be an integral domain, $G$ an abelian group and $\alpha \equiv \unita$. Then the following implication holds:
\begin{center}
(i) $I \cap \tilde{\aalg} \neq \{0\}$, for every non-zero two-sided ideal $I$ in $\gcrossedring$. \\
$\Downarrow$ \\
(ii) $\tilde{\aalg}$ is a maximal commutative subring in $\gcrossedring$.
\end{center}
\end{thm}

\begin{proof}
This follows from Corollary \ref{maxabeliansigmaconverse} and Corollary \ref{idealtoaction}.
\end{proof}

\begin{exmp}[The quantum torus]
Let $q \in \complex \setminus \{0, 1\}$ and denote by \\ $\complex_q[x,x^{-1},y,y^{-1}]$ the \emph{twisted Laurent polynomial ring} in two non-commuting variables under the twisting
\begin{eqnarray}\label{qformula}
y \, x = q \, x \, y.
\end{eqnarray}
The ring $\complex_q[x,x^{-1},y,y^{-1}]$ is known as the \emph{quantum torus}. Now let
\begin{itemize}
\item $\aalg = \complex[x,x^{-1}]$,
\item $G = (\integers,+)$,
\item $\sigma_n : P(x) \mapsto P(q^n x)$, $n \in G$,
\item $\alpha(s,t) = \unita$ for all $s,t \in G$.
\end{itemize}
It is easily verified that $\sigma$ and $\alpha$ together satisfy conditions (i)-(iii) of $G$-systems and it is not hard to see that $\gcrossedring \cong \complex_q[x,x^{-1},y,y^{-1}]$. In the current example, $\aalg$ is an integral domain, $G$ is abelian, $\alpha \equiv \unita$ and hence all the conditions of Theorem \ref{idealtomaxcomm} are satisfied. Note that the commutation relation \eqref{qformula} implies
\begin{eqnarray}\label{qmulti}
y^n \, x^m = q^{mn} \, x^m \, y^n, \quad \forall n,m \in \integers.
\end{eqnarray}
It is important to distinguish between two different cases:
\begin{case}[$q$ is a root of unity]
Suppose that $q^n = 1$ for some $n\neq 0$. From equality \eqref{qmulti} we note that $y^n \in Z(\complex_q[x,x^{-1},y,y^{-1}])$ and hence $\complex[x,x^{-1}]$ is not maximal commutative in $\complex_q[x,x^{-1},y,y^{-1}]$. Thus, according to Theorem \ref{idealtomaxcomm}, there must exist some non-zero ideal $I$ which has zero intersection with $\complex[x,x^{-1}]$.
\end{case}
\begin{case}[$q$ is not a root of unity]
Suppose that $q^n \neq 1$ for all $n \in \integers \setminus \{0\}$. One can show that this implies that $\complex_q[x,x^{-1},y,y^{-1}]$ is \emph{simple}. This means that the only non-zero ideal is $\complex_q[x,x^{-1},y,y^{-1}]$ itself and this ideal obviously intersects $\complex[x,x^{-1}]$ non-trivially. Hence, by Theorem \ref{idealtomaxcomm}, we conclude that $\complex[x,x^{-1}]$ is maximal commutative in $\complex_q[x,x^{-1},y,y^{-1}]$.
\end{case}
\end{exmp}

\section{Ideals, intersections and zero-divisors}

Let $D$ denote the subset of zero-divisors in $\aalg$ and note that $D$ is always non-empty since $\zeroa \in D$. By $\tilde{D}$ we denote the image of $D$ under the embedding $\iota$.


\begin{thm}\label{idealszerodivisors}
Let $\aalg$ be commutative. Then the following implication holds:
\begin{center}
(i) $I \cap \left( \tilde{\aalg} \setminus \tilde{D} \right) \neq \emptyset$, for every non-zero two-sided ideal $I$ in $\gcrossedring$. \\
$\Downarrow$ \\
(ii) $D \cap \aalg^G = \{ \zeroa \}$, i.e. the only zero-divisor that is fixed under all automorphisms is $\zeroa$.
\end{center}
\end{thm}

\begin{proof}[Proof by contrapositivity]
Let $\aalg$ be commutative. Suppose that $D \cap \aalg^G \neq \{ \zeroa \}$. Then there exist some $c \in D \setminus \{\zeroa\}$ such that $\sigma_s(c)=c$ for all $s \in G$. There is also some $d \in D \setminus \{\zeroa\}$, such that $c \cdot d = \zeroa$. Consider the ideal
\begin{displaymath}
\ann(c) = \{a \in \aalg \, \mid \, a\cdot c = \zeroa \}
\end{displaymath}
of $\aalg$. It is clearly non-empty since we always have $\zeroa \in \ann(c)$ and $d \in \ann(c)$. Let $\theta$ be the quotient homomorphism
\begin{displaymath}
\theta : \aalg \to \aalg / \ann(c), \quad a \mapsto a + \ann(c).
\end{displaymath}
Let us define a map $\rho : G \to \aut(\aalg / \ann(c))$ by
\begin{displaymath}
\quad \rho_s(a + \ann(c)) = \sigma_s(a) + \ann(c)
\end{displaymath}
for all $a + \ann(c) \in \aalg/\ann(c)$. Note that $\ann(c)$ is invariant under $\sigma_s$ for all $s \in G$ and thus it is easily verified that $\rho$ is a well-defined automorphism on $\aalg/\ann(c)$. By introducing the function
\begin{displaymath}
\beta : G \times G \to U(\aalg/\ann(c)), \quad (s,t) \mapsto (\theta \circ \alpha)(s,t)
\end{displaymath}
it is easy to verify that $\{\aalg/\ann(c),G,\rho,\beta\}$ is in fact a $G$-crossed system. Now consider the map

\begin{displaymath}
\Gamma : \gcrossedring \to \aalg/\ann(c) \rtimes_\beta^{\rho} G, \quad \sum_{s \in G} a_s \overline{s} \mapsto \sum_{s \in G} \theta(a_s) \overline{s}.
\end{displaymath}
For any two elements $\sum_{s \in G} a_s \overline{s}, \sum_{t \in G} b_t \overline{t} \in \gcrossedring$ the additivity of $\Gamma$ follows by

\begin{eqnarray*}
\Gamma \left(\sum_{s \in G} a_s \overline{s} + \sum_{t \in G} b_t \overline{t} \right) &=& \Gamma \left(\sum_{s \in G} (a_s + b_s) \overline{s} \right) = \sum_{s \in G} \theta(a_s + b_s) \overline{s} \\
&=& \sum_{s \in G} \theta(a_s) \overline{s} + \sum_{t \in G} \theta(b_t) \overline{t} = \Gamma \left(\sum_{s \in G} a_s \overline{s} \right) + \Gamma \left(\sum_{t \in G} b_t \overline{t} \right)
\end{eqnarray*}
and due to the assumptions, the multiplicativity follows by
\begin{eqnarray*}
\Gamma \left(\sum_{s \in G} a_s \overline{s} \sum_{t \in G} b_t \overline{t} \right) &=& \Gamma \left(\sum_{(s,t) \in G \times G} a_s \, \sigma_s(b_t) \, \alpha(s,t) \, \overline{st} \right) \\
&=& \sum_{(s,t) \in G \times G} \, \theta(a_s \, \sigma_s(b_t) \, \alpha(s,t) ) \, \overline{st} \\
&=& \sum_{(s,t) \in G \times G} \, \theta(a_s) \, \theta(\sigma_s(b_t)) \, \theta(\alpha(s,t)) \, \overline{s t} \\
&=& \sum_{(s,t) \in G \times G} \theta(a_s) \, \rho_s( \theta(b_t) ) \, \beta(s,t) \, \overline{s t} \\
&=& \left(\sum_{s \in G} \theta(a_s) \overline{s} \right) \left(\sum_{t \in G} \theta(b_t) \overline{t} \right) = \Gamma \left(\sum_{s \in G} a_s \overline{s} \right) \, \Gamma \left(\sum_{t \in G} b_t \overline{t} \right)
\end{eqnarray*}
where have used that $\beta = \theta \circ \alpha$ and $\theta(\sigma_s(b_t)) = \rho_s( \theta(b_t))$ for all $b_t \in \aalg$ and $s\in G$. This shows that $\Gamma$ is a ring homomorphism. Now, pick some $g \neq e$ and let $I$ be the ideal generated by $d\, \overline{g}$. Clearly $I \neq \{0\}$ and we see that $\Gamma\lvert_{I} \equiv 0$. Note that $\ker(\theta) = \ann(c)$ and in particular we have
\begin{displaymath}
\Gamma(a \overline{e}) = 0 \Longrightarrow a \in \ann(c).
\end{displaymath}
Take $m \, \overline{e} \in I \cap \left( \tilde{\aalg} \setminus \tilde{D} \right)$. Then $\Gamma(m \, \overline{e}) = 0$ and hence $m \in \ann(c) \subseteq D$, which is a contradiction. This shows that
\begin{displaymath}
I \cap \left( \tilde{\aalg} \setminus \tilde{D} \right) = \emptyset
\end{displaymath}
and by contrapositivity this concludes the proof.
\end{proof}

\begin{exmp}[The truncated quantum torus]
Let $q \in \complex \setminus \{0, 1\}$, $m \in \naturalnumbers$ and consider the ring
\begin{eqnarray*}
\frac{\complex[x,y,y^{-1}]}{(y \, x - q \, x \, y \, , \, x^m)}
\end{eqnarray*}
which is commonly referred to as the \emph{truncated quantum torus}. It is easily verified that this ring is isomorphic to $\gcrossedring$ with
\begin{itemize}
\item $\aalg = \complex[x]/(x^m)$,
\item $G = (\integers,+)$,
\item $\sigma_n : P(x) \mapsto P(q^n x)$, $n \in G$,
\item $\alpha(s,t) = \unita$ for all $s,t \in G$.
\end{itemize}
One should note that in this case $\aalg$ is commutative, but not an integral domain. In fact, the zero-divisors in $\complex[x]/(x^m)$ are precisely those polynomials where the constant term is zero, i.e. $p(x)=\sum_{i=0}^{m-1} a_i \, x^i$, with $a_i \in \complex$, such that $a_0=0$. It is also important to remark that, unlike the \emph{quantum torus}, $\gcrossedring$ is never simple (for $m>1$). In fact we always have a chain of two-sided ideals
\begin{displaymath}
\frac{\complex[x,y,y^{-1}]}{(y \, x - q \, x \, y \, , \, x^m)} \supset \langle x \rangle \supset \langle x^2 \rangle \supset \ldots \supset \langle x^{m-1} \rangle \supset \{0\}
\end{displaymath}
independent of the value of $q$. Moreover, the two-sided ideal $J = \langle x^{m-1} \rangle$ is contained in $\comm(\complex[x]/(x^m))$ and contains elements outside of $\complex[x]/(x^m)$. Hence we conclude that $\complex[x]/(x^m)$ is not maximal commutative in $\frac{\complex[x,y,y^{-1}]}{(y \, x - q \, x \, y \, , \, x^m)}$. When $q$ is a root of unity, with $q^n=1$ for some $n < m$, we are able to say more. Consider the polynomial $p(x) = x^n$, which is a non-trival zero-divisor in $\complex[x]/(x^m)$. For every $s \in \integers$ we see that $p(x)=x^n$ is fixed under the automorphism $\sigma_s$ and therefore, by Theorem \ref{idealszerodivisors}, we conclude that there exists a non-zero two-sided ideal in $\frac{\complex[x,y,y^{-1}]}{(y \, x - q \, x \, y \, , \, x^m)}$ such that its intersection with $\tilde{\aalg} \setminus \tilde{D}$ is empty.
\end{exmp}

\section{Comments to the literature}

The literature contains several different types of intersection theorems for group rings, Ore extensions and crossed products. Typically these theorems rely on heavy restrictions on the coefficient rings and the groups involved. We shall now give references to some interesting results in the literature.

It was proven in \cite[Theorem 1, Theorem 2]{Rowen} that the center of a semiprimitive (semisimple in the sense of Jacobson \cite{Jacobson}) P.I. ring respectively semiprime P.I. ring has a non-zero intersection with every non-zero ideal in such a ring. For crossed products satisfying the conditions in \cite[Theorem 2]{Rowen}, it offers a more precise result than Theorem \ref{killerproof} since $Z(\gcrossedring) \subseteq \comm(\tilde{\aalg})$. However, every crossed product need not be semiprime nor a P.I. ring and this justifies the need for Theorem \ref{killerproof}.

In \cite[Lemma 2.6]{LorenzPassman} it was proven that if the coefficient ring $\aalg$ of a crossed product $\gcrossedring$ is prime, $P$ is a prime ideal in $\gcrossedring$ such that $P \cap \tilde{\aalg} = 0$ and $I$ is an ideal in $\gcrossedring$ properly containing $P$, then $I \cap \tilde{\aalg} \neq 0$. Furthermore, in \cite[Proposition 5.4]{LorenzPassman} it was proven that the crossed product $\gcrossedring$ with $G$ abelian and $\aalg$ a $G$-prime ring has the property that, if $G_{\inn}=\{e\}$, then every non-zero ideal in $\gcrossedring$ has a non-zero intersection with $\tilde{\aalg}$.
It was shown in \cite[Corollary 3]{FisherMontgomery} that if $\aalg$ is semiprime and $G_{\inn}=\{e\}$, then every non-zero ideal in $\gcrossedring$ has a non-zero intersection with $\tilde{\aalg}$.
In \cite[Lemma 3.8]{LorenzPassman2} it was shown that if $\aalg$ is a $G$-prime ring, $P$ a prime ideal in $\gcrossedring$ with $P\cap \tilde{\aalg}=0$ and if $I$ is an ideal in $\gcrossedring$ properly containing $P$, then $I\cap \tilde{\aalg} \neq 0$. In \cite[Proposition 2.6]{MontgomeryPassman} it was shown that if $\aalg$ is a prime ring and $I$ is a non-zero ideal in $\gcrossedring$, then $I \cap (\aalg \rtimes_{\alpha}^{\sigma} G_{\inn}) \neq 0$. In \cite[Proposition 2.11]{MontgomeryPassman} it was shown that for a crossed product $\gcrossedring$ with prime ring $\aalg$, every non-zero ideal in $\gcrossedring$ has a non-zero intersection with $\tilde{\aalg}$ if and only if $C^t[G_{\inn}]$ is $G$-simple and in particular if $|G_{\inn}| < \infty$, then every non-zero ideal in $\gcrossedring$ has a non-zero intersection with $\tilde{\aalg}$ if and only if $\gcrossedring$ is prime.

Corollary \ref{maxcommcorollary} shows that if $\tilde{\aalg}$ is maximal commutative in $\gcrossedring$, without any further conditions on the coefficient ring or the group, we are able to conclude that every non-zero ideal in $\gcrossedring$ has a non-zero intersection with $\tilde{\aalg}$.

In the theory of group rings (crossed products with no action or twisting) the intersection properties of ideals with certain subrings have played an important role and are studied in depth in for example \cite{FormanekLichtman}, \cite{LorenzPassman3} and \cite{Passman2}. Some further properties of  intersections of ideals and homogeneous components in graded rings have been studied in for example \cite{CohenMontgomery}, \cite{MarubNauweOysta}.

For ideals in Ore extensions there are interesting results in \cite[Theorem 4.1]{Irving1D} and \cite[Lemma 2.2, Theorem 2.3, Corollary 2.4]{LaunLenaRiga}, explaining a correspondence between certain ideals in the Ore extension and certain ideals in its coefficient ring. Given a domain $\aalg$ of characteristic $0$ and a non-zero derivation $\delta$ it is shown in \cite[Proposition 2.6]{Irving2D} that every non-zero ideal in the Ore extension $R= \aalg [x; \delta]$ intersects $\aalg$ in a non-zero $\delta$-invariant ideal. 
Similar types of intersection results for ideals in Ore extension rings can be found in for example \cite{LeroyMatczuk} and \cite{McConnellRobson}.

The results in this article appeared initially in the preprint \cite{OinSil}.\\

\noindent \textbf{Acknowledgements.} We are grateful to Marcel de Jeu, Christian Svensson, Theodora Theohari-Apostolidi and especially Freddy Van Oystaeyen for useful discussions on the topic of this article.

\end{document}